\documentclass[12pt]{amsart}

\usepackage{amsfonts, amssymb, amscd}
\usepackage{verbatim}
\usepackage{eucal}
\usepackage{amssymb}
\usepackage{mathrsfs}
\usepackage{graphicx}
\usepackage{psfrag}
\usepackage{cite}
\usepackage{upref}

\def\dual                 {{\vee}}
\def\rk                 {{\rm rk}}

\def\Pic		{{\rm Pic}}

\def\ZZ                 {{\mathbb Z}}
\def\PP                {{\mathbb P}}
\def\RR                 {{\mathbb R}}
\def\CC                 {{\mathbb C}}

\newtheorem{lemma}{Lemma}[section]
\newtheorem{theorem}[lemma]{Theorem}
\newtheorem{corollary}[lemma]{Corollary}
\newtheorem{proposition}[lemma]{Proposition}
\theoremstyle{definition}
\newtheorem{definition}[lemma]{Definition}

\newtheorem{conjecture}[lemma]{Conjecture}

\newtheorem{remark}[lemma]{Remark}
\theoremstyle{remark}
\newtheorem*{proof*}{Proof}
\numberwithin{equation}{section}
\begin{document}
\title[On $(2,4)$ complete intersections]{On $(2,4)$ complete intersection threefolds that contain an Enriques surface}

\author{Lev A.  Borisov}
\address{Rutgers University, Department of Mathematics, 110 Frelinghuysen Rd.,
Piscataway \\ NJ \\ 08854 \\ USA}
\email{borisov@math.rutgers.edu}
\author{Howard J. Nuer}
\address{Rutgers University, Department of Mathematics, 110 Frelinghuysen Rd.,
Piscataway \\ NJ \\ 08854 \\ USA}
\email{hjn11@math.rutgers.edu}

\begin{abstract}
We study nodal complete intersections threefolds of type $(2,4)$
in $\PP^5$ which contain an Enriques surface
in its Fano embedding. We completely determine Calabi-Yau
birational models of a generic such threefold. These models have Hodge
numbers $h^{11}=2,h^{12}=32$. We also describe Calabi-Yau varieties
with Hodge numbers $(h^{11},h^{12})$ equal to $(2,26)$, $(23,5)$ and $(31,1)$.
The last two pairs of Hodge numbers are, to the best of our knowledge, new.
\end{abstract}

\maketitle

\section{Introduction}\label{sec1}
A smooth 
projective threefold $X$ is called Calabi-Yau if it satisfies $\omega_X\cong \mathcal O_X$ and 
$h^1(\mathcal O_X)=h^2(\mathcal O_X)=0$.
Such threefolds have attracted considerable attention over the years as possible 
targets for type II superstring compactifications. From the mathematical point of view, 
Calabi-Yau threefolds
are natural higher dimensional analogs of the celebrated K3 surfaces. However, their geometry 
is much more varied. 

It is not known in general whether Calabi-Yau threefolds fall into a finite or infinite number 
of different families; current constructions have already produced over 30,000
diffeomorphism classes. The vast majority of these have appeared as hypersurfaces in
Fano toric varieties from the work of Skarke and 
Kreuzer (\cite{KS97},\cite{KS00}), based on the reflexive polytope construction of Batyrev \cite{Bat94}.
In this paper we use a non-toric method of constructing Calabi-Yau threefolds, inspired 
by the work of Gross and Popescu \cite{GP}. Specifically, in \cite{GP} the authors considered Calabi-Yau threefolds that contain
$(1,d)$ polarized abelian surfaces for small $d$. They constructed their threefolds as small
resolutions of nodal complete intersections with equations contained in the ideal of 
the abelian surface. 

In our paper we consider Calabi-Yau $(2,4)$ complete intersections in $\PP^5$ that contain
an Enriques surface in its Fano embedding. Explicitly, these varieties can be constructed as
follows. Consider a four-dimensional complex vector space $V$. Take a generic 
four-dimensional subspace $W$ of $V^\dual\otimes V^\dual$ of bilinear forms on $V$.
Then $X=X_W$ is a hypersurface in the Grassmannian ${\rm Gr}(2,V)$  which consists
of two-dimensional subspaces $V_1\subseteq V$ such that the space of restrictions of elements of $W$ to
$V_1^\dual\otimes V_1^\dual$ has dimension at most three. 

This variety $X$ contains an Enriques surface under its Reye embedding.
It has $58$ ODP singularities at those $V_1$ where the dimension of the restriction
of $W$ is two and admits two small resolutions, $X^0$ and $X^1$. The Calabi-Yau threefold 
$X^0$ has the structure of a fibration over $\PP^1$ with the above-mentioned Enriques surface appearing 
as the reduction of a double fiber.  The Calabi-Yau variety $X^1$ has a small contraction
onto a complete intersection $\hat X$ of type $(4,4)$ in a weighted projective space 
$\PP(1,1,1,1,2,2)$ with variables $(u_1,\ldots,u_4,y,z)$ given explicitly by the equations
$$
\left\{
\begin{array}{lcl}
y^2&=&\det(\sum_i u_iA_i)\\
z^2&=&-\det(\sum_i u_i(A_i+B_i))+(y-{\rm Pf}(\sum_i u_iB_i))^2 
\end{array}
\right.
$$
where $A_i$ and $B_i$ are symmetric and antisymmetric parts of a basis of $W$.

Two more models of $X$ are obtained by means of an involution $\sigma$ on $\hat X$
which changes the sign of $z$, see Figure 1. Their Hodge numbers are 
 $(h^{11},h^{12})=(2,32)$ \hskip-3pt\footnote{We note that while this construction is new, these hodge numbers have been discovered previously in \cite{BK}.}. These exhaust all minimal models of $X$, which is the 
 main result of this paper.

The paper is organized as follows. In Section \ref{background} we recall the 
basic results about Enriques surfaces and their Fano embeddings.  In Section \ref{sec2}
we introduce our Calabi-Yau varieties. We give an alternative description of them as 
determinantal hypersurfaces in the Grassmannian. In Section \ref{secwp} we begin to
investigate the birational geometry of our varieties. Specifically, we produce a birational
model which is a nodal complete intersection in a weighted projective space.
We finish the discussion of birational geometry of our varieties in Section \ref{bir}.
We state some open questions in Section \ref{seccom}. 
In this section we also describe a construction of Calabi-Yau varieties with novel
Hodge numbers $(23,5)$ and $(31,1)$.
We chose to collect some of the technical statements that we use along the way 
in the appendices to streamline the main exposition.

{\bf Acknowledgements.} We thank Igor Dolgachev, Mark Gross, and Mihnea Popa for useful discussions.  We are also grateful to Mike Stillman and Dan Grayson for the program \textit{Macaulay2} \cite{GS} which was instrumental to our project.  Additionally, we thank an anonymous referee for various small corrections as well as useful help with the additional arguments and references required to remove the use of computer calculations in the proof of Theorem \ref{construction}.  The authors were partially supported by NSF Grant DMS  1201466.

\section{Background: Enriques surfaces}\label{background}
We present here a survey of pertinent results about Enriques surfaces.  For further information about Enriques surfaces see \cite{CD} or Section 5 of \cite{DM}.  Unless otherwise mentioned,
 proofs of facts about Enriques surfaces can be found there. 

\begin{definition} A smooth complex projective surface $S$ is called an Enriques surface if 
 $h^1(S,\mathcal O_S)=h^2(S,\mathcal O_S)=0$
 and the canonical bundle satisfies $\omega_S^2\cong \mathcal O_S$.
\end{definition}

\subsection{The Picard lattice and Fano models}\label{Fano Models}

It is well known that the Picard group of an Enriques surface $S$, denoted $\Pic(S)$, coincides with its Neron-Severi group and has torsion subgroup generated by the canonical class $K_S$.  Furthermore, ${\rm Num}(S)=\Pic(S)/(K_S)$, its group of divisors modulo numerical equivalence, comes with a nondegenerate symmetric bilinear form defined by the intersection pairing.  It makes ${\rm Num}(S)$ a rank 10 lattice with signature $(1,9)$.  It follows from the classification of even unimodular lattices that it must be isomorphic to $\mathbb E=U\oplus E_8$, where $U$ is the hyperbolic plane and $E_8$ is the unique negative definite even unimodular lattice of rank 8.  In \cite{DM} one sees that $\mathbb E$ can be realized as a primitive sublattice of the odd hyperbolic lattice \[\mathbb Z^{1,10}=\mathbb Z e_0+...+\mathbb Z e_{10}\] with $e_0^2=1,e_i^2=-1$ for $i>0$.  One can show that $\mathbb E$ is isomorphic to the orthogonal complement of the vector \[k_{10}=-3e_0+e_1+...+e_{10}.\]  We form the vectors \[f_i=e_i-k_{10}=3e_0-\sum_{j=1,j\neq i}^{10} e_j, i=1,...,10,\] which lie in $k_{10}^{\perp}$, and thus in this copy of $\mathbb E$.  These vectors form what's called an isotropic 10-sequence, that is a sequence of 10 isotropic vectors such that $f_i\cdot f_j=1$ for $i\neq j$.  Defining \[\Delta=10e_0-3e_1-...-3e_{10},\] we get that \[3\Delta=f_1+...+f_{10}.\]  This construction is very important when considered in ${\rm Num}(S)$.  We refer to \cite{DM} for the definition of a \textit{canonical} isotropic 10-sequence and the proof of the following 
\begin{theorem} For any canonical isotropic 10-sequence $(f_1,...,f_{10})$ in ${\rm Num}(S)$, there exists a unique $\delta \in \overline{\rm Amp}(S)\cap {\rm Num}(S)$ such that \[3\delta=f_1+...+f_{10}.\]  It satisfies $\delta^2=10,\delta\cdot f\geq 3$ for any nef isotropic class $f$.  Conversely, any nef $\delta$ satisfying this property can written as above for some canonical isotropic 10-sequence defined uniquely up to permutation.
\end{theorem}

It is well known that nef primitive isotropic vectors in the Picard lattice come from nef effective divisors $F\in \Pic(S)$ with $F^2=0$ and such that $|2F|$ is an elliptic pencil with precisely two double fibers, $F$ itself and its conjugate $F+K_S$.  

Let $\Delta\in \Pic(S)$ be a divisor whose numerical equivalence class is as in the theorem above.  Then this defines a nef effective divisor with $\Delta^2=10$ and $\Delta\cdot F\geq 3$ for any nef effective divisor $F\in \Pic(S)$ with arithmetic genus 1.  From Lemma 4.6.2 in \cite{CD}, the linear system $|\Delta|$ induces a birational morphism $S\rightarrow \mathbb P^5$ onto a normal surface $\overline{S}$ of degree 10 with at worst nodal singularities.  In this case $\overline{S}$ is called a \textit{Fano model} of $S$ and $\Delta$ is called a \textit{Fano polarization}.  One defines the degeneracy invariant of the isotropic 10-sequence and finds that $\Delta$ is ample if and only if it is defined by a non-degenerate isotropic 10-sequence, which is the generic case.  The corresponding half-fibers $F_i$ are mapped to plane cubics lying in 10 planes $\Lambda_1,...,\Lambda_{10}$ such that $\Lambda_i\cap\Lambda_j\neq \varnothing$.  Moreover the conjugates $F_{-i}=F_i+K_S$ lie in 10 different planes $\Lambda_{-1},...,\Lambda_{-10}$.  The intersection of planes corresponding to one 10-sequence with the planes in the conjugate sequence depends on the geometry of the specific Fano embedding.  For the beautiful geometry surrounding these 20 planes see \cite{DM}.  We refer to \cite{BP} for another discussion of isotropic sequences, Fano models, and their relations with the automorphism group of an Enriques surface.

The existence of Fano polarizations was proved in \cite{CP}.

\subsection{Reye and Cayley Models of nodal Enriques surfaces}

It is not easy to describe the Fano embeddings of general Enriques surfaces. However,
such descriptions are more readily available for surfaces in the codimension one family of the 10-dimensional moduli space of Enriques surfaces consisting of so-called nodal \hskip-3pt\footnote{The term \emph{nodal} here means the existence of a $(-2)$ curve, rather than 
the presence of a singularity in the Fano embedding, see \cite{CD}.} Enriques surfaces.

Let $V$  be a complex vector space of dimension four. Let $W_+$ be a four-dimensional
subspace of ${\rm Sym^2}(V^\dual)$.  We can view it as a three-dimensional linear system (called a web) of quadrics in $\PP V$.   By considering the corresponding bilinear forms one also gets a 
three-dimensional linear system $W'$ of symmetric $(1,1)$ divisors on $\PP V \times \PP V$.
One says that $W_{+}$ is a \textit{regular} web if the base locus of $W'$, denoted $Y_{W_{+}}$, is smooth.  From the adjunction formula one sees that in this case $Y_{W_{+}}$ is a K3 surface.  One can check that for a regular web $W_{+}$ the base locus of $W_{+}$ is empty and the involution $\tau$ switching the factors has no fixed points along $Y_{W_{+}}$.  Thus the quotient of $Y_{W_{+}}$ by $\tau$ is an Enriques surface $S_{W_{+}}$ with K3 cover $Y_{W_{+}}$.  

Consider the $(1,1)$ line bundle on $\PP V\times \PP V$. It restricts to a line bundle of degree $20$ 
on $Y_{W_+}$. Its pushforward to $S_{W_+}$ splits into a direct sum of two line bundles of 
degree $10$ each, which correspond to different lifts of the involution $\tau$ to the $(1,1)$ bundle.
Depending on the choice of the line bundle its sections are symmetric or skew forms in
$H^0(\PP V\times \PP V,{\mathcal O}(1,1))=V^\dual\times V^\dual$.

\subsubsection{Reye Models}\label{Reye models}
The linear system $|\bigwedge^2 V^{\dual}|\subset H^0(\PP V\times \PP V,{\mathcal O}(1,1))$ restricted to $Y_{W_{+}}$ defines a morphism \[\pi:Y_{W_{+}}\rightarrow \mathbb P(\bigwedge^2 V),\] which factors through $S_{W_+}$, where $\pi(x,y)=\overline{xy}$, the line through $x$ and $y$ for $(x,y)\in \mathbb P^3\times\mathbb P^3\backslash\Delta$. The image of this morphism is a smooth surface $Re(W_{+})$ contained in $G(2,V)$ in its Pl\"{u}cker embedding and isomorphic to $S_{W_{+}}$.  This model of the nodal Enriques surface $S_{W_+}$ is called the \textit{Reye model}.  Consult Theorem 5.1 in \cite{DM} for an explicit geometric description of the plane cubics and (-2)-curves in this model and for a proof that this is indeed a Fano polarization.  It is a celebrated result that an Enriques surface is nodal if and only if it is isomorphic to a Reye model.  This was proved for a generic nodal Enriques by Cossec in \cite{CR} and for any nodal Enriques by Dolgachev and Reider, see \cite{DR},\cite{CD2}.

\subsubsection{Cayley Models}
Consider now  the linear system of symmetric $(1,1)$ divisors
$|{\rm Sym}^2(V^{\dual})|\subset H^0(\PP V\times \PP V,{\mathcal O}(1,1))$. It
 induces a closed embedding \[\phi:(\mathbb P(V)\times \mathbb P(V))/\tau \rightarrow \mathbb 
 P({\rm Sym}^2(V)).\]  Explicitly, this map sends $(\CC x,\CC y)$ 
 to $\CC (x\otimes y+y\otimes x)$.  The restriction to $S_{W_{+}}$ gives an embedding of $S_{W_{+}}$ in $\mathbb P({\rm Sym}^2(V))$.  In fact, the image lies in $\PP^5\cong
  \PP({\rm Ann(W_+)})\subset \PP({\rm Sym}^2(V))$.
 This image is denoted by $Ca(W_{+})$ and called the \textit{Cayley model} of the  Enriques surface $S_{W_+}$.  From this description we easily see that $Ca(W_{+})$ is the
locus of elements in $\PP({\rm Ann(W_+)})\subset \PP({\rm Sym}^2(V))$
which have rank $2$ when viewed  as quadrics in $\PP V^\dual$.

Clearly the Reye and Cayley models, being isomorphic to the same Enriques surface, are isomorphic, and the above constructions suggest that the defining polarizations are related.  Indeed one has the following

\begin{proposition} Let $|\Delta^r|$ be the linear system defining the Reye model of $S_{W_{+}}$.  Then the Cayley model of $S_{W_{+}}$ is defined by the linear system $|\Delta^c|:=|\Delta^r+K_{S_{W_{+}}}|$.
\end{proposition}

For us the most relevant difference between the Cayley and Reye models is that the Reye model is contained in a quadric, while the Cayley model is not.  
\subsection{Equations of Fano models}\label{Fano equations}
Using the descriptions of the Cayley and Reye models above, we get explicit descriptions of the defining equations of these Enriques surfaces in their Fano embeddings. 

\begin{proposition} As above, let $Re(W_{+})\subset \mathbb P(\bigwedge^2 V)$ and $Ca(W_{+})\subset \mathbb P(\rm Ann(W_+))$ be the Reye and Cayley models of the nodal Enriques surface $S_{W_+}$, respectively.  Then the linear system of cubics containing these models (in the respective projective space) is 9-dimensional.  Moreover, the homogeneous ideal of $Ca(W_+)$ is generated by these 10 cubics, while the homogeneous ideal of $Re(W_+)$ is generated by a quadric and 4 cubics.
\end{proposition}
\begin{proof} We sketch parts of the proof here while deferring to \cite{DM} for the details.  If $S\subset \mathbb P^5$ is any Fano model defined by an ample Fano polarization $\Delta$, then $\Delta^2=10$.  By Serre duality and Kodaira vanishing applied to $\Delta+K_S$, we find that $h^0(S,3\Delta)=46$ from Riemann-Roch theorem.  But $h^0(\mathbb P^5, \mathcal O_{\mathbb P^5}(3))=56$, so we must have at least 10 cubics in the homogeneous ideal of $S\subset \mathbb P^5$.  From \cite{GLM} it follows that Fano models are 3-normal so that $h^1(\mathbb P^5,\mathcal I_S(3))=0$, and thus $h^0(\mathbb P^5,\mathcal I_S(3))=10$.  

In the special case of Reye and Cayley Fano models, we can describe  the 10 cubics explicitly.  For $Re(W_+)$, we can account for 6 of the cubics as products of the quadric defining $G(2,4)$, which contains $Re(W_+)$, with a linear polynomial.  We will give an explicit description of  4 additional cubics in Section 3.2.  For the Cayley model, notice that in the preceding subsection we showed that $Ca(W_+)=\phi(\mathbb P^3\times \mathbb P^3) \cap \mathbb P({\rm Ann}( W_+))$.  Recall that $\phi(\mathbb P^3\times \mathbb P^3)$ was the locus of reducible quadrics, which is the singular locus of the quartic hypersurface  $\mathcal D_4$
in $\mathbb P^9=\mathbb P({\rm Sym}^2(V))$ parametrizing singular quadrics.  Thus $Ca(W_+)$ is the intersection of the ten cubic partials of $\mathcal D_4$'s defining equation with the generic 5-plane $\mathbb P({\rm  Ann}(W_+))$.  Again, see \cite{DM} for proofs that the homogeneous ideals are generated as claimed.
\end{proof}

For unnodal Enriques surfaces the situation is similar to that of the Cayley model.
\begin{proposition} The homogeneous ideal of a Fano model of a general Enriques surface is generated by 10 cubics.
\end{proposition}
\begin{proof} In \cite{CR} it's shown that a Fano model of a generic Enriques surface is not contained in a quadric.  The rest is discussed in \cite{DM}.
\end{proof}

\section{Complete intersection of type (2,4) that contains a Reye Enriques and its small resolutions}\label{sec2}
For the bulk of the paper we consider the following geometric data.

Let $V$ be a complex vector space of dimension four. Let $W$ be 
a generic four-dimensional subspace in $V^\dual\otimes V^\dual$. 
It will be often convenient to identify $V$ with the space of column vectors and think of $W$ as generated by four linearly independent $4\times 4$ matrices. We will
further separate the symmetric and skew parts of the matrices so that
$W$ is generated by $A_i+B_i$, $i=1,\ldots,4$ with $A_i=A_i^*$, $B_i=-B_i^*$.

\subsection{Reye Models}
To $W\subset V^\dual\otimes V^\dual$ above, we naturally associate a web of quadrics 
$W_+$ on $\PP V$
generated by the symmetric parts of elements of $W$ (i.e. by $A_i$ in our explicit presentation).

We further assume for the rest of this section that $W_+$ induces a regular web of quadrics in $\mathbb P^3=\mathbb P(V)$ (for the definition and properties of such
regular webs see Section \ref{Reye models}), and consider the Reye model of the generic 
nodal Enriques surface, denoted $Re(W_+)$. 
Recall that $W_+$ defines a 3-dimensional linear system 
of symmetric $(1,1)$ divisors on $\PP^3
\times \PP^3$ whose base locus is a smooth K3 surface $Y_{W_+}$.  
The Enriques surface $S_{W_+}$ is defined as the quotient of $Y_{W_+}$ by the 
involution that interchanges the copies of $\PP^3$.

There is a morphism $Y_{W_+} \rightarrow
\mathbb P(\bigwedge^2 V)$ 
which factors through the Enriques surface $S_{W_+}$ and defines a closed embedding 
 $$\pi:S_{W_+}\to \PP(\Lambda^2V).$$ 
 The
projective model $Re(W_+)$ is the image of $S_{W_+}$, under this morphism.
It is a smooth Enriques surface in $\mathbb P^5$ contained in the Grassmannian $G(2,4)$ under the Pl\"ucker embedding.  
The embedding
$\pi$ corresponds to a polarization $\Delta$ of degree 10 on $S_{W_+}$, and in fact $Re(W_+)$ is a Fano model of the
generic nodal Enriques surface $S_{W_+}$ [see Sections \ref{Reye models} and \ref{Fano Models}].

From Section \ref{Fano equations} we know that the homogeneous ideal of $Re(W_+)$ is generated by 
the Pl\"ucker quadric $Q$ defining $G(2,4)$ and four cubics $C_1,\ldots,C_4$. 
 Let $X$ be a (2,4) complete
intersection cut out by $Q$ and a generic element $P$ of degree four in the ideal 
of $Re(W_+)$.

\begin{theorem}\label{construction} For general $P$ we have

(1) $X$ is an irreducible threefold whose singular locus consists of 58 ordinary double points
all of which lie on $Re(W_+)$.

(2) There is a small resolution $\pi_1:X^1\rightarrow X$ 
of the ordinary double points, with $X^1$ a
non-singular Calabi-Yau threefold, obtained by blowing up $X$ along $Re(W_+)$.

(3) There is another small resolution $\pi_0:X^0 \rightarrow X$ 
of the ordinary double points, with
$X^0$ 
a non-singular Calabi-Yau threefold, and such that there is a map 
$X^0\rightarrow \PP^1$ 
whose generic fiber is a K3 surface. It contains $Re(W_+)$ as a double fiber.

(4) $\chi(X^1)=\chi(X^0)=-60$, 
$h^{1,1}(X^1)=h^{1,1}(X^0)=2$, and
$h^{2,1}(X^1)=h^{2,1}(X^0)=32$. \end{theorem}

\begin{proof} (1) Since $X$ is a complete intersection, it is certainly connected.  To see that its singular
locus is as claimed for general $P$, we first restrict ourselves to the singular locus along $Re(W_+)$, that is
$Sing(X)\cap Re(W_+)$.  We note that the equations defining $X$, $Q$ and $P$, define sections of
$(\mathcal I/\mathcal I^2)(2)$ and $(\mathcal I/\mathcal I^2)(4)$, respectively, where $\mathcal I$ denotes the
ideal sheaf of $Re(W_+)$ in $\mathbb P^5$.  A local calculation shows that $Sing(X)\cap Re(W_+)$ is precisely the
subscheme where these two sections are linearly dependent.  To calculate the degree of this subscheme we note that
using the standard exact sequence for tangent and normal bundles and the Euler exact sequence on $\mathbb P^5$ we
find that $c((\mathcal I/\mathcal I^2)(4))=1+6H+138[{\rm pt}]$ up to numerical equivalence, where $H$ is the hyperplane
class on $Re(W_+)$.  The section $Q$ of $(\mathcal I/\mathcal I^2)(2)$ induces the following short  exact
sequence \[0\rightarrow \mathcal O(2)\rightarrow (\mathcal I/\mathcal I^2)(4)\rightarrow V\rightarrow 0.\] Taking
Chern classes we find that $c_2(V)=58$ is precisely the degree of the subscheme where these two sections are
linearly dependent.  So we would expect to have 58 singularities along $Re(W_+)$ counted with multiplicities.  
However, a \textit{Macaulay2} calculation gives examples for which $X$ has precisely 58 ordinary double
points, and thus this is true for general parameter choices, for example by semi-continuity of Milnor numbers.  Since the singular locus has dimension 0, we see that indeed $X$
is an irreducible threefold.

For the sake of completeness, we mention a direct geometric proof that the singular locus is as claimed for generic choice of parameters.  Consider the restriction of the linear system of cubics containing the Enriques surface $Re(W_+)$ to the Grassmanian $G(2,4)$ containing it.  Then $Re(W_+)$ is the scheme-theoretic base locus of this linear system.  This is true also of the linear system $\Lambda$ of quartics containing $Re(W_+)$.  Let $\pi:\tilde{G}\rightarrow G$ be the blow up of $G=G(2,4)$ along $Re(W_+)$ with exceptional divisor $E$.  Then $|\pi^*\Lambda-E|$ is base-point free, and by Bertini's theorem a generic member $\tilde{X}$ of this linear system is nonsingular.  Thus the strict transform of $X$, the intersection with $G(2,4)$ of a generic quartic containing $Re(W_+)$, is nonsingular.  Moreover, according to Bertini's theorem the singular locus of $X$ occurs along the base-locus of the linear system, namely along $Re(W_+)$.

According to Theorem 2.1 and Claim 2.2 of \cite{DH}, the singular locus of $X$ has codimension 2 on $Re(W_+)$, and these finite number of points are precisely the points of $X$ over which $\tilde{X}$ contains an entire fiber of $\pi$.  Now restricting $|\pi^*\Lambda-E|$ to $E$ preserves base-point freeness and thus the intersection with $E$ of a generic member of this linear system is a nonsingular surface $S$ by Bertini.  Let $C\subset \tilde{X}$ be one of the fibers of $\pi$ contracted to a singular point of $X$.  Then consider the normal bundle exact sequence, $$0\rightarrow N_{C\subset S}\rightarrow N_{C\subset \tilde{X}}\rightarrow N_{S\subset \tilde{X}}|_C\rightarrow 0.$$  Since $S$ is a smooth surface mapped birationally to smooth $Re(W_+)$ by $\pi$, $N_{C\subset S}$ is $\mathcal O_{\mathbb P^1}(-1)$.  Furthermore, since $S$ is the intersection of an element of $|\pi^*\Lambda-E|$ with $|E|$, the normal bundle $N_{S\subset \tilde{X}}$ is the restriction of $\mathcal O_{\tilde{X}}(E)$ to $E$ and then to $S$.  But since $E$ is the exceptional fibre of the blow-up of a smooth variety along a smooth subvariety, $\mathcal O_{\tilde{X}}(E)|_E\cong \mathcal O_E(-1)$.  Thus $N_{S\subset \tilde{X}}=\mathcal O_S(-1)$.  Restricting to $C$ gives $\mathcal O_{\mathbb P^1}(-1)$.  Since $\text{Ext}^1(\mathcal O_{\mathbb P^1}(-1),\mathcal O_{\mathbb P^1}(-1))=0$, we see that $N_{C\subset \tilde{X}}$ must be $\mathcal O_{\mathbb P^1}(-1)\oplus\mathcal O_{\mathbb P^1}(-1)$.  

It follows from (5.13, (b)) in \cite{Pagoda} that these curves are contracted to ordinary double points.  The number of nodes is then calculated as above, using the Chern class calculation.

(2 and 3) According to Lemma \ref{small res}, blowing up $X$ along $Re(W_+)$ gives a small resolution
$X^1\rightarrow X$ of the ordinary double points with $X^1$ a nonsingular Calabi-Yau threefold.  
Simultaneously flopping the 58 exceptional $\mathbb P^1$'s we obtain a second small resolution $X^0$.  Since
$Re(W_+)$ contained the 58 nodes, the blow-up when restricted to $Re(W_+)$ is simply the usual blow-up of smooth
points, and therefore flopping these exceptional curves has the effect of blowing down the (-1)-curves on
$\pi_1^{-1}(Re(W_+))$ (see (5.13) of \cite{Pagoda}).  Therefore, $X^0$ contains a copy of $Re(W_+)$, and we get the claimed fibration by Proposition 
\ref{CYfibration}.

(4) By Batyrev's theorem on the birational invariance of Hodge numbers for Calabi-Yau's \cite{Bat99}, it suffices to check
these claims on $X^1$.  The claim about the topological Euler characteristic follows immediately from the
fact that a generic (2,4) complete intersection has Euler characteristic -176, and from
the number of nodes.  
The calculation of the Hodge numbers follows from the results of Section \ref{Hodge} and \textit{Macaulay2} calculations.  An alternative calculation of the Hodge numbers will be given in Remark \ref{alternative hodge}.
\end{proof}

\begin{remark} It was observed in  \cite{CR} that an Enriques surface in its Fano embedding in $\PP^5$ 
is contained in a quadric if and only if this surface is a Reye Enriques. Thus our construction describes a general
CY $(2,4)$ intersection in $\PP^5$ that contains an Enriques surface in its Fano embedding. 
\end{remark}

\subsection{Alternate description as a determinantal variety}\label{alternate}

We offer here a different, but very useful description of $X$ as a determinantal variety
inside the Grassmannian $G(2,V)$. The Reye model $Re(W_+)$ is then the set of lines $l
\subset \PP V = \mathbb P^3$ such that the image of ${\rm Span}(\{A_i\}_{i=1,...,4})$ under the restriction map to the subspace $\CC^2\subset V$ corresponding
to $l$ has dimension at most 2  (Remark 5.5 in \cite{DM}).  The cubics containing $Re(W_+)$ modulo the ideal of the Grassmannian can be described
very explicitly.  For a given plane $\mathbb C^2\subset \mathbb C^4$ representing our line $l\in \PP V$
choose a basis \[v=\left(\begin{matrix} a_1 \\ a_2 \\a_3 \\a_4 \end{matrix}\right), w=\left( \begin{matrix} b_1
\\b_2\\b_3\\b_4\end{matrix}\right),\] and consider the $4 \times 3$ matrix

 \begin{equation}\label{matrix}
 \left( \begin{matrix} v^* A_1 v & v^*
A_1 w & w^* A_1 w\\ v^* A_2 v & v^* A_2 w & w^* A_2 w\\ v^* A_3 v & v^* A_3 w & w^* A_3 w\\ v^* A_4 v & v^* A_4 w &
w^* A_4 w \end{matrix} \right).
\end{equation}
 The four $3 \times 3$ minors of this matrix give degree 6 equations in the
$a_i,b_i$, and writing these minors in terms of the Pl\"{u}cker coordinates of $l$ gives 4 cubic equations, defined modulo the ideal of the Grassmannian.  These are the 4 cubics that cut out $Re(W_+)$ inside $G(2,4)$.

To present the determinantal description of $X$ we need the following lemma: 
\begin{lemma} 
The natural map
\[H^0({\rm Gr}(2,4),\mathcal O_{{\rm Gr}(2,4)}(1))\otimes H^0({\rm Gr}(2,4),\mathcal I(3))\rightarrow H^0({\rm Gr}(2,4),\mathcal I(4)),\]
given by multiplication, is an isomorphism, where $\mathcal I$ is the ideal sheaf of $Re(W_+)$ in $G(2,4)$. 
\end{lemma}

\begin{proof} One may verify this for generic $W_+$ using \textit{Macaulay2}. \end{proof}

Now from the lemma we see that the choice of a quartic hypersurface in $\mathbb P^5$ that contains $Re(W_+)$ is
a choice 
of a sum of products of a linear form and a $3\times 3$ minor of the matrix above.  But a linear form
$l(x_0,...,x_5)$ on $\mathbb P^5=\mathbb P(\bigwedge^2 V)$ corresponds to a skew-form $B$ on $V$, whose
restriction to the plane $\mathbb C^2$ is determined by $v^* Bw$.  Thus we have four skew-forms $B_1,..,B_4$, and
we may express our quartic as \[\det \left( \begin{matrix} v^* B_1 w&v^* A_1 v & v^* A_1 w & w^* A_1 w\\
v^* B_2 w&v^* A_2 v & v^* A_2 w & w^* A_2 w\\
v^* B_3 w&v^* A_3 v & v^* A_3 w & w^* A_3 w\\
v^* B_4 w&v^* A_4 v & v^* A_4 w & w^* A_4 w \end{matrix}\right ) .\]

Thus we get the following description:

\begin{proposition} 
Let $Q$ be the universal quotient bundle on $G(2,4)$.  Take a generic dimension four
subspace of sections of  $Q\otimes Q$.  Then the nodal Calabi-Yau threefold $X$ 
is 
the determinantal variety on $G(2,4)$ given by \[\det \left(\begin{matrix} s_1 \\
s_2\\s_3\\s_4\end{matrix}\right).\]
\end{proposition}

\begin{proof} Indeed the only thing left to note is that $X$ was cut out on the Grassmannian by a quartic
containing $Re(W_+)$.  The choice of such a quartic was seen to be given by a $4\times 4$ determinant like the one
preceding the proposition.  Each row combined the restrictions of a skew-symmetric and symmetric matrix to a $\CC^2\subset \CC^4$ in such a way that we get
a section of $Q\otimes Q$.  This gives the desired description. \end{proof}

\begin{remark}
Observe that $H^0({\rm Gr}(2,V),Q\otimes Q)$ is naturally isomorphic to $V^\dual\otimes V^\dual$.
Moreover, the above description of the subspace means that it can be viewed as the 
subspace $W$ of $V^\dual\otimes V^\dual$ considered in the beginning of the section.
\end{remark}

\begin{remark}\label{scaleskew}
It is clear that one can scale the $B_i$-s simultaneously without affecting the determinantal
variety. In invariant terms, this is due to the fact that $Q\otimes Q$ is the direct sum of 
symmetric and skew parts. The corresponding naive parameter count for the moduli space of $X$, 
$\dim {\rm Gr}(4,V^\dual\otimes V^\dual)-\dim PGL(V) - 1 = 48-15-1=32$, is in fact correct from Theorem \ref{construction}.
\end{remark}

\subsection{The fibration structure} In this subsection we will give an explicit way of 
seeing the fibration structure on $X^0$. In the process we will see that the K3 fibers
of this fibration are determinantal quartics in $\PP^3$.

We again start by considering a space $W\subseteq V^\dual\otimes V^\dual$.
We will identify $V$ with the column vectors of length four for convenience.
Consider the intersection inside $\PP^3\times \PP^3\times \PP^1$ of the 
four $(1,1,1)$ divisors given by 
$$
x^*(sA_i+tB_i)y=0,~{\rm for~all~}i.
$$
Here $x$ and $y$ denote column vectors and $(s:t)\in\PP^1$.
Note that this space contains the irreducible component $\Delta=\{(x,x,(0:1))|x\in \PP^3\}$. Denote by $Z$ 
the closure of the complement of this component in $\PP^3\times \PP^3\times \PP^1$.  A \textit{Macaulay2} calculation shows that $Z$ is smooth.  Observe that $Z$ is preserved by the involution $\tau$ that sends
$$\tau:(x,y,(s:t))\mapsto (y,x,(s:-t)).$$ 

\begin{remark}
It is standard that the variety $Z$ is fibered over $\PP^1$ with fibers that are isomorphic to determinantal 
quartics in $\PP^3$. The fiber over $(1:0)$ is the symmetric K3 surface $Y_{W_+}$ which
is the double cover of $S_{W_+}$ considered in \ref{Reye models}. The involution $\tau$ acts 
on the base of this fibration, so that $Z/\tau$ maps to $\PP^1$. The 
reduction of the fiber of this map 
over $(1:0)$ is the Enriques surface $S_{W_+}$.
\end{remark}

\begin{theorem}
The quotient space $Z/\tau$ is naturally birational to $X$.
\end{theorem}

\begin{proof}
Let us define the two birational maps which are inverses of each other.
First we describe the map from $Z/\tau$ to $X$. The point
$$
(x,y,(s:t))
$$
is sent to $\CC x\oplus \CC y\subset\CC^4$ which gives 
a point in ${\rm Gr}(2,V)$.  Clearly, this map is well-defined modulo $\tau$ away from $Z\cap \Delta$.

Observe that the restrictions of $sA_i+tB_i$ 
to $V_1=\CC x\oplus \CC y\subset\CC^4$ lie in the codimension one subspace of $V_1^{\vee}\otimes V_1^{\dual}$ defined by $x^*(-)y=0$. As a consequence, the image lies in the determinantal variety in $G(2,V)$ given by the $sA_i+tB_i$. In view of Remark \ref{scaleskew},
this implies that the image lies in $X$.

In the opposite direction, let $V_1\subseteq V$ be a dimension two subspace such
that ${\rm Span}(A_i+B_i)$ restricts to a proper subspace of $V_1^\dual\otimes V_1^\dual$. We may assume $V_1$ is generic  in $X$ so that the space of restrictions is of dimension three. Consider a generator of its annihilator  $l\in V_1\otimes V_1$. Split it into symmetric and skew-symmetric parts $l=l_++l_-$. Then $tl_++sl_-$ annihilates $sA_i+tB_i$ for all $i$ and $(s:t)\in \PP^1$. For a generic point $V_1$ corresponding to a point in $X$ and generic $(s:t)\in \PP^1$, the tensor $tl_++sl_-$ will be indecomposable.  The space of decomposable tensors
is a quadric in $\PP(V_1\otimes V_1)$. Consider $(s:t)$ such that $tl_++sl_-$ is decomposable,
i.e. $tl_++sl_-=x\otimes y$. We see that $tl_+-sl_-=y\otimes x$, and these are the 
only decomposable linear combinations of $l_+$ and $l_-$. We then map
the point of $X$ that corresponds to $V_1$ to $(x,y,(s:t))$ for any decomposable
$tl_++sl_-=x\otimes y$, with the choice of a decomposable linear combination irrelevant
modulo $\tau$.

It is easy to see that these two constructions are inverses of each other.
\end{proof}

\begin{remark}
Observe that the line through $l_+$ and $l_-$ lies in the quadric of decomposable
tensors if and only if $l_-=0$ and $l_+$ is decomposable, i.e. $l_+=x\otimes x$.
This means that $x^*A_ix=0$ for all $i$ which does not happen for a $W$ with regular web $W_+$.  Thus the only ambiguity in the above map occurs at the points where the dimension of ${\rm Span}(A_i+B_i)|V_1$ drops by $2$ (i.e. along $Re(W_+)$). In this case, one needs to get a choice of $l$ that annihilates all $A_i+B_i$. This means that the map lifts to a map from $X^0$.
\end{remark}

\begin{remark} Based upon computational evidence, we believe that $Z/\tau$ is in fact isomorphic to $X^0$ blown up at the two points corresponding to the two-dimensional subspaces $\CC x\oplus \CC y$ such that $x^*B_i y=0$ for all $i$. 
\end{remark}

\section{A birational model of $X$ as a nodal complete intersection in  the weighted projective 
space}\label{secwp}
In this section we continue to investigate the birational geometry of $X^1$. We find that it maps to
a certain nodal complete intersection of type $(4,4)$ inside a weighted projective space 
$\PP(1,1,1,1,2,2)$. As before, we consider four generic $4\times 4$ symmetric (resp. skew) matrices
$A_i$ (resp. $B_i$).

Let $(u_1:u_2:u_3:u_4: y : z)$ be homogeneous coordinates on the weighted projective space
 $\PP(1,1,1,1,2,2)$ of the respective
weights. Consider the complete intersection of two degree $4$ hypersurfaces in this weighted projective
space defined by the property that 
\begin{equation}\label{formal}
(y+z q + {\rm Pf}(\sum_i u_iB_i)q^2)(y-z q +{\rm Pf}(\sum_i u_iB_i)q^2)=\det(\sum_i u_i(A_i+qB_i))
\end{equation}
as polynomials in the formal variable $q$.
Note that the right hand side is an even polynomial in $q$. The coefficient at $q^0$ is
$\det(\sum_i u_iA_i)$ and the coefficient at $q^4$ is $\det(\sum_i u_iB_i) = ({\rm Pf}(\sum_i u_iB_i))^2$.
The middle coefficient is therefore $\det(\sum_i u_i(A_i+B_i))-\det(\sum_i u_iA_i)-({\rm Pf}(\sum_i u_iB_i))^2$.
The equation \eqref{formal} thus reduces to two equations
\begin{equation}\label{ciwp}
\left\{
\begin{array}{lcl}
y^2&=&\det(\sum_i u_iA_i)\\
z^2&=&-\det(\sum_i u_i(A_i+B_i))+(y-{\rm Pf}(\sum_i u_iB_i))^2 
\end{array}
\right.
\end{equation}
of degree $4$ in $\PP(1,1,1,1,2,2)$.
\begin{definition}
Denote by $\hat X$ the complete intersection defined by \eqref{ciwp}.
\end{definition}

\begin{theorem}\label{wpmodel}
Let $H=\pi_1^*\mathcal O_{X}(1)$ and $E$ the class of $\pi_1^{-1}(Re(W_+))$ inside 
$N^1(X^1)$, the Neron-Severi group of $X^1$.  The ray $\RR_{\geq 0}(3H-E)$ in $\tilde{N}^1(X^1)=N^1(X^1)\otimes \mathbb R$ defines a small contraction
of $X^1$ to the complete intersection $\hat X$ in $\PP(1,1,1,1,2,2)$ defined above.  
\end{theorem}

\begin{proof}
We first calculate the intersection product on $N^1(X^1)$. We have $\int_{X^1}H^3=8$, $\int_{X^1}H^2E=10$, $\int_{X^1}HE^2=0$,
$\int_{X^1}E^3=-58$. The last statement follows because $E$ is the blowup of an Enriques surface at $58$ points
and $E\vert_E$ is the sum of the $58$ exceptional lines of this blowup.
We also have from the adjunction formula
$12+58=\int_E c_2(E)=\int_{X^1} E\frac {1+c_2(X^1)+c_3(X^1))}{(1+E)}=\int_{X^1} c_2(X^1)E+E^3=\int_{X^1}c_2(X^1)E-58$ so that $\int_{X^1} c_2(X^1)E=128$.  We calculate from the Riemann-Roch formula that $\chi(H)=\int_{X^1} \frac 16 H^3 +\frac 1{12}c_2(X^1)H$, and since $H$ is nef, Kawamata-Viehweg vanishing gives $\chi(H)=h^0(X^1,H)=6$. Thus we get $\int_{X^1}c_2(X^1)H = 12(6-\frac 86)=56$.

Consider the divisor $L=3H-E$ on $X^1$. Notice that $\int_{X^1}L^3=4$. We also have $\int_{X^1}c_2(X^1)L=40$. Thus, $\chi(kL)=\frac 23k^3+\frac {10}3k$, so in particular $\chi(L)=4$, $\chi(2L)=12$. 
We can view $X^1$ as a divisor on the blowup $\widehat{Gr}$
of ${\rm Gr}(2,4)$ along $Re(W_+)$, by viewing it as the proper preimage
of $X$ under the blowup.  Since the ideal of $Re(W_+)$ on $G(2,4)$ is generated by four cubics, see \ref{Fano equations}, the divisor $3H-E$ is base-point-free on $\widehat{Gr}$ and thus on $X^1$.  Consequently, $3H-E$ is nef, thus $h^0(L)=4$, $h^0(2L)=12$.

We can describe the global sections of $L$ and $2L$ in more detail. As in 
subsection \ref{alternate}
we can describe the sections of $L$ as $3\times 3$ minors of the $4\times 3$ matrix of restrictions 
\eqref{matrix}. For a point in ${\rm Gr}(2,V)$ given by a vector space $V_1$ these minors $u_i$ are such 
that $\sum_i u_iA_i$ restricts to $0$ on $V_1$. In fact, $\widehat{Gr}$ is naturally embedded into
${\rm Gr}(2,V)\times \PP^3$ and the map by $\vert L\vert$ is simply the projection to the second factor
(restricted to $X^1\subset \widehat{Gr}$).  Points in $X^1$ can be thought of as 
two-dimensional subspaces $V_1\subset V$ together with a choice of $(u_1:\ldots:u_4)\in \PP^3$ 
such that $\sum_iu_i(A_i+B_i)$ restricts trivially on $V_1$. Note that this implies $\sum_i u_i(A_i+qB_i)$
restricts trivially to $V_1$ for any $q$.

Pick a basis of $V_1$ to represent it by a $4\times 2$ matrix $T$. We have 
$$
T^*\sum_i u_i (A_i+qB_i) T = {\bf 0}.
$$
We will now use the results of Section  \ref{linalg}.
Consider for given $W$, $V_1$, $u$, and $q$,
$$
\det (S^*\sum_i u_i (A_i+qB_i) T )
$$
as a function of $S$, where $S$ is a $4\times 2$ matrix. If the dimension of the span of columns of $S$ and $T$ is less than four, then there is a linear combination of columns of $S$ that lies in that of $T$. This implies that there is a linear combination of the rows of 
$S^*\sum_i u_i (A_i+qB_i) T$ which is zero, thus the determinant above vanishes. Since this determinant is a 
polynomial function of $S$, it must be divisible by the irreducible polynomial $\det(T\vert S)$, where $T\vert S$ is the $4\times 4$ matrix obtained by juxtaposing $T$ and $S$. Since the $S$-degree
of both polynomials is $2$, the ratio depends only on $W$, $V_1$, $u$, and $q$. The dependence on $q$ is 
clearly quadratic. Moreover, the coefficient by $q^2$ is equal to ${\rm Pf}(\sum u_i B_i)$, see 
Proposition \ref{skew}. Thus 
we have 
$$
\det (S^*\sum_i u_i (A_i+qB_i) T) = (\hat y+\hat zq+{\rm Pf}(\sum u_i B_i)q^2)\det(T\vert S).
$$
for some $\hat y$ and $\hat z$.
By taking a transpose and switching $q\to -q$ we get
$$
\det (T^*\sum_i u_i (A_i+qB_i) S )= (\hat y-\hat zq+{\rm Pf}(\sum u_i B_i)q^2)\det(T\vert S).
$$
By Proposition \ref{tn} this implies 
$$
(\hat y+\hat zq+{\rm Pf}(\sum u_i B_i)q^2)(\hat y-\hat zq+{\rm Pf}(\sum u_i B_i)q^2) = \det (\sum_i u_i(A_i+qB_i)).
$$
Let us now investigate the dependence of $\hat y$ and $\hat z$ on the choice of the basis
of $V_1$. A different choice of basis leads to a matrix $T_1=TU$ for some $2\times 2$
invertible matrix $U$. Observe that $u_i$ are sections of $L=3H-E$ and can be 
naturally viewed as sections of $3H$, i.e. they scale by  $(\det U)^3$.
By scaling $S\mapsto SU$ we see that $\hat y$ and $\hat z$ 
scale by $(\det U)^6$. They thus descend from functions on the frame bundle over 
$X^1$ to sections of $6H$. 

We now observe that $\hat y$ and $\hat z$ are sections of $2L=6H-2E$, because 
they are integral over $\CC[u_1,\ldots,u_4]$ and $\oplus_{n\geq 0}H^0(X^1,nL)$
is integrally closed. As a result, $(u_1,...,u_4,\hat y,\hat z)$ give a map from $X^1$ to the complete intersection $\hat X$. 
We have checked using \textit{Macaulay2} that $\hat X$ is nodal with $42$ ordinary double points. It is also easy to
see that $X^1\to \hat X$ has three-dimensional image. Consequently, the subring of $\oplus_n H^0(X^1,nL)$
generated by $u_i$, $\hat y$ and $\hat z$ has the same Hilbert series as a complete intersection of type $(4,4)$
in $\PP(1,1,1,1,2,2)$, i.e. $H(t)=\frac {(1+t^2)^2}{(1-t)^4}$. Since this coincides with the dimensions of 
$H^0(X^1,nL)$ obtained via Riemann-Roch and Kawamata vanishing, we see that $u_1,\ldots,u_4,\hat y,\hat z$
generate $\oplus_n H^0(X^1,nL)$, and $X^1\to \hat X$ is a contraction that corresponds to the ray
$\RR_{\geq 0}(3H-E)$. Observe that this divisor is not ample, because $\hat X$ is singular. 
\end{proof}

\begin{remark}
It is instructive to see explicitly the birational map from 
$\hat X$ to $X^1$ which is the inverse of $X^1\to \hat X$. Let $(u_1,\ldots,u_4,y,z)$ satisfy \eqref{ciwp}.
We are working birationally and can thus assume that ${\rm Pf}(\sum_i u_iB_i)\neq 0$. 
Consider two roots $q_1$ and $q_2$ of 
$$y+zq+{\rm Pf}(\sum_i u_iB_i)q^2=0$$ 
From \eqref{ciwp} we see that $\sum_i u_i(A_i+q_1B_i)$ 
is degenerate. Consider its right kernel, call it $x_1\in V$. Similarly, we can consider the right kernel $x_2$
of $\sum_i u_i(A_i+q_2B_i)$. Then the point in $X^1$ is given (generically) by ${\rm Span}(x_1,x_2)$. Indeed these elements lie in $V_1$ by Proposition \ref{kernel}.
\end{remark}

\begin{remark}
The map from $\widehat {Gr}$ defined by multiples of $3H-E$ is interesting in its own right.
One can show that it maps onto the hypersurface in $\PP(1,1,1,1,2)$ given by 
$$
y^2=\det(\sum_i u_iA_i).
$$
The generic fibers are rational curves which correspond to a ruling of the smooth quadric $\sum_i u_iA_i$ in $\PP V$.  The choice of sign in $y$ (locally) distinguishes the two rulings.  
\end{remark}

\begin{remark}\label{alternative hodge}
From the preceding remark and the proof of the preceeding theorem, we may now present a purely geometric, non-computer-based proof that the Picard rank of $X^1$ is 2.  The blow-up $\widehat{Gr}$ of $Gr(2,4)$ along the Enriques surface $Re(W_+)$ necessarily has Picard rank 2 as the blow-up of a smooth variety of Picard rank 1 along a smooth subvariety.  We've seen that $3H-E$, in the notation above, is base-point free and thus nef.  Moreover, since we know that $3H-E$ is nef but not ample, and likewise for $H$, we see that $4H-E$ lies in the interior of the cone spanned by these two rays and thus must be ample.  By the Lefschetz hyperplane theorem it follows that $X^1$ also has Picard rank 2 as a generic member of the ample linear system $|4H-E|$.  
\end{remark}
\section{Birational models of $X^1$}\label{bir}
It is known that the Picard group $\Pic(X)$ of a smooth Calabi-Yau variety $X$ is isomorphic to its Neron-Severi group $N^1(X)$.  Its rank $\rho(X)$ is called the Picard number.   We denote the nef cone of $X$ by $\overline{\rm Amp}(X)$.  It consists of those Cartier $\mathbb R$-divisors that have non-negative intersection with every effective curve in $X$.  Its relative interior is the ample cone ${\rm Amp}(X)$, the convex cone generated by ample divisors.  We also write $\overline{\rm Mov}(X)$ for the convex cone generated by the set of movable divisors, effective divisors that move in a linear system without fixed components.  Its interior is denoted ${\rm Mov}(X)$.  Since a birational map between Calabi-Yau varieties $X$ and $Y$ would have indeterminacy locus of codimension at least two, their divisor vector spaces $\tilde{N}^1(X)$ and $\tilde{N}^1(Y)$ are identified.  From the following lemma (Lemma 1.5 in \cite{Kaw97}) we see that their nef cones coincide under this identification if and only if the birational map between them is an isomorphism, and otherwise they are disjoint:
\begin{lemma}\label{kaw}
Let $X$ be a smooth Calabi-Yau variety, and denote its birational automorphism group by $Bir(X)$.  
Then $g\in Bir(X)$ is a biregular automorphism if and only if there exists an ample divisor $H$ on $X$ such that
$g^*H$ is ample.
\end{lemma}
On the other hand, this isomorphism between $\tilde{N}^1(X)$ and $\tilde{N}^1(Y)$ identifies $\overline{\rm Mov}(X)$ and $\overline{\rm Mov}(Y)$.  Moreover, for a Calabi-Yau threefold $X$ it follows from the existence and termination of flops (see \cite{Kaw88}) that the movable cone is covered by the nef cones of all birational models of $X$ that are isomorphic outside of a locus of codimension at least 2.  This arrangement is called the moveable fan.  

We now consider the vector space $\tilde{N}^1(X^1)=N^1(X^1)\otimes \mathbb R$ for the smooth Calabi-Yau varieties $X^1$ constructed above.  It is a real vector space of dimension 2 by Theorem \ref{construction}. We uncovered in Theorem \ref{construction} two smooth birational models $X^0$ and $X^1$ of $X$.  On $X^0$ we found the divisor $E$ whose ray $\mathbb R_{\geq 0} E$ induced a fibration of $X^0$ over $\mathbb P^1$ so that $E$ must be nef but clearly not ample.  It thus forms an extremal ray of $\overline{\rm Amp}(X^0)$.  The other extremal ray is generated by $H$ which is nef as the pull back from $X$ of the nef (even very ample) line bundle $\mathcal O_X(1)$, but which cannot be ample as it contracts the 58 exceptional $\mathbb P^1$'s.  The variety $X^1$, the flop of $X^0$ by these curves, has $H$ as one of its extremal rays for the same reason.  As we noted in the proof of Theorem \ref{wpmodel}, the divisor $3H-E$ is nef but not ample so it forms the other extremal ray.  We can now complete the analysis and describe explicitly the entire movable cone and all of the other smooth birational models.

First note that by two results of Kawamata on a Calabi-Yau threefold $X$, any $\phi\in Bir(X)$ can be decomposed into a sequence of flops followed by an automorphism of $X$ at the last stage, and any flop of $X$ corresponds to an extremal ray of $\overline{\rm Amp}(X)$ (see Theorem 5.7 in \cite{Kaw08} and Theorem 1 in \cite{Kaw88}, respectively).  Thus $\mathbb R_{\geq 0} E$ forms an extremal ray of the entire movable cone since it induces a fibration, which cannot be flopped.

Now consider the map from $X^1\to \hat X$. It is a resolution of $42$ nodes of $\hat X$. We can flop it to obtain another resolution $X^2\to \hat X$. Note also that $\hat X$ admits an involution 
$$
\sigma: (u_1:\ldots:u_4:y:z)\mapsto (u_1:\ldots:u_4:y:-z).
$$
By a \textit{Macaulay2} calculation, there are $22$ singular points of $\hat X$ that are fixed by this 
involution (the remaining $20$ are pairs of points where $y=0$ and $\sum u_iA_i$ is a rank $2$ quadric on $V$, i.e. they lie over the Hessian surface of $W_+$ which has precisely 10 nodes \cite{DM}).  Note that $\sigma$ lifts to a birational automorphism of $X^1$ and thus acts 
on $N^1(X^1)$.

\begin{proposition}
The action of $\sigma$ on $N^1(X^1)$ is nontrivial. Specifically,
it sends $$(H,E)\mapsto (17H-6E, 48H-17E).$$
\end{proposition}

\begin{proof}
If this map was trivial, then it would act as an automorphism of $X^1$ by Lemma \ref{kaw}.
However, it does not extend to a morphism on any of the $22$ exceptional lines 
of $X^1\to \hat X$ whose images are fixed by $\sigma$ because on $\hat X$ it switches the two resolutions locally around the fixed nodes.

Clearly, $L=3H-E$ is fixed by the involution since this divisor corresponds to $\mathcal O_{\hat{X}}(1)$ and the involution comes from one on $\mathbb P(1,1,1,1,2,2)$.  Because it acts 
nontrivially, the other eigenvalue is $(-1)$ and we have $H\mapsto k(3H-E)-H$ for some $k$.  Since the degree of $H$ as a subvariety of $\PP(1,1,1,1,2,2)$ must be preserved by the involution, the intersection number $H(3H-E)^2$ must be preserved.  This implies $k=6$ and finishes the proof.
\end{proof}

We observe that  $X^2$ is actually isomorphic to $X^1$ via the lift of the involution on $\hat X$.  
Specifically, the birational map $g: X^1 \rightarrow \hat{X}\rightarrow \hat{X}\dashrightarrow X^2$, obtained by applying the involution on $\hat{X}$ in the middle, maps the nef cone $\overline{\rm Amp}(X^1)$ to an adjacent cone, which must then be the cone of $X^2$, since both contain $L$.
As a consequence, the nef cone of $X^2$ is generated by $3H-E$ and $17H-6E$.  The contraction that corresponds to $17H-6E$ contracts $58$ smooth $\PP^1$s.  These can be flopped to get a $K3$ fibration via $96H-34E$, with $48H-17E$ a smooth Enriques surface which is the reduction of a double fiber of the fibration.  This gives the

\begin{corollary}
The birational models of $X$ look as in Figure 1.  Moreover, $\overline{\rm Mov}(X)$ is the convex cone generated by $E$ and $48H-17E$.
\begin{figure}[tbh]
\begin{center}
\includegraphics[scale = .6]{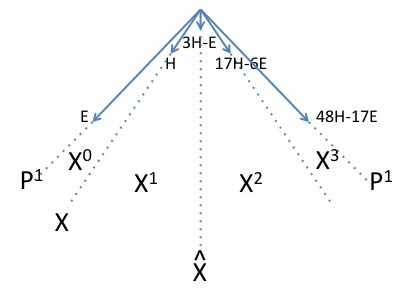}
\end{center}
\caption{Birational models of $X$}
\end{figure}
Here $X^{i}$ are depicted inside their nef cones. The picture is slightly distorted for readability.
\end{corollary}

It seems worth mentioning that the above result verifies the Kawamata-Morrison cone conjecture for this family (see \cite{Kaw97},\cite{Mor} for statements and discussion).


\section{Comments}\label{seccom}
\subsection{Relation to other constructions of Calabi-Yau threefolds.}
Batyrev and Kreuzer have obtained Calabi-Yau threefolds with Hodge numbers $(2,32)$ 
in \cite{BK} in the context of degenerations of toric complete intersections. We do not know at the moment how many distinct families are constructed there, and whether our construction yields one of those families. 

\subsection{Other complete intersections in $\PP^5$ that contain Enriques surfaces.}
It is only natural to investigate other kinds of complete intersections that contain Enriques surfaces.  In particular, if we have an Enriques surface $S$ in its Fano embedding, we may want to look at a generic $(3,3)$ complete intersection $X$ in $\PP^5$ that contains it. It can be shown using methods analogous to those above that such a complete intersection has $48$ ODP's that lie on the Enriques surface. Its blowup $X^1$ along $S$ and the flop $X^0$ of the 48 exceptional $\mathbb P^1$'s are somewhat similar to the ones considered in this paper.  They have Hodge 
numbers $h^{11}=2$, $h^{12}=26$. This again fits one of the Hodge pairs from \cite{BK}.

We are, however, unable to describe the movable cones for generic such (3,3) complete intersections.  The divisor $E$ corresponding to the preimage of $S$ under the blow-up still induces a fibration of $X^0$ by K3 surfaces with double fiber along $S$, and thus this divisor still forms an extremal ray of the movable cone.  Moreover, since the Fano model of a generic Enriques surface is not contained in a quadric, its homogeneous ideal is generated by 10 cubics.  Thus on the blow-up along $S$, the divisor $3H-E$ is base-point free and thus nef.  When $S$ is a generic nodal surface either in its Cayley or Reye embedding, the divisor $3H-E$ induces a divisorial contraction on $X^1$. In the Reye case there is a surface on $X^1$ which gets 
contracted to a genus one curve of singularities  on the image.  For the Cayley model, the proper preimage of the union of the trisecants of $S$ that sit inside $X$ form a surface which gets contracted.  Hence the divisor $3H-E$ is not ample and forms an extremal ray of the movable cone.  For a generic unnodal Enriques surface, however, we do not have as explicit a description for the cubics that cut it out.  While we do not have an explicit description of the entire movable cone for (3,3) complete intersections coming from generic unnodal Enriques surfaces, we do have the following positive result:

\begin{proposition}  For the $X^1$ birational model of (3,3) complete intersection Calabi-Yau's coming from generic unnodal Enriques surfaces, the divisor $3H-E$ is ample.
\end{proposition}
\begin{proof} Take as a starting point the $X^1$ model for a complete intersection coming from a generic Reye Enriques.  Then as mentioned above the divisor $3H-E$ contracts a smooth ruled surface $S$ over a curve of genus 1.  By Proposition 4.1 in \cite{W} the locus $\Gamma$ of deformations of $X^1$ for which $S$ deforms in the family is a smooth divisor.  By taking a 1-parameter family of deformations transversal to $\Gamma$ and coming from deforming the nodal Reye Enriques to be a generic unnodal one, Proposition 4.4 in \cite{W} shows that the class $3H-E$ will be ample in some punctured neighborhood of the origin of this 1-parameter family.
\end{proof}

\subsection{Mirrors.} We have been unable to find mirrors of our varieties. However, in
the process of looking for them, we have constructed Calabi-Yau threefolds with novel
Hodge numbers $(23,5)$ and $(31,1)$ as follows. 
Motivated by \cite{HT}, we consider a subfamily given by matrices
$A_i+B_i$ with 
$$
A_1+B_1=\left(
\begin{array}{cccc}
a&b&0&0\\
c&d&0&0\\
0&0&0&0\\
0&0&0&0\\
\end{array}
\right),~~
A_2+B_2=\left(
\begin{array}{cccc}
0&0&0&0\\
0&a&b&0\\
0&c&d&0\\
0&0&0&0\\
\end{array}
\right),
$$
$$
A_3+B_3=\left(
\begin{array}{cccc}
0&0&0&0\\
0&0&0&0\\
0&0&a&b\\
0&0&c&d\\
\end{array}
\right),~~
A_4+B_4=\left(
\begin{array}{cccc}
d&0&0&c\\
0&0&0&0\\
0&0&0&0\\
b&0&0&a\\
\end{array}
\right)
$$
with generic $a,b,c,d$.  Calculations in \textit{Macaulay2} suggest that a generic member of this
family gives a $(2,4)$ complete intersection in $\PP^5$ with $106$ ODPs. Interestingly,
the Hodge number calculations yield a larger than obvious five-dimensional family,
with the Hodge numbers $(23,5)$. We have verified the existence of crepant resolutions of these nodal complete intersections.

If we further restrict the subfamily by picking $a=d$, then we get complete intersections
with $118$ ODPs, and we have shown the existence of crepant resolutions of these singular points.  The calculations of Section \ref{Hodge} then yield $(31,1)$ Hodge
numbers. This family might be mirror to the known families with Hodge numbers $(1,31)$
constructed in \cite{Ton} and \cite{Kap}. We hope to study this family further in a subsequent
paper.

\section{Appendix 1: Hodge numbers calculations}\label{Hodge}
Here we present the method we used to calculate the Hodge numbers of the Calabi-Yau varieties we construct.  We present the method for general complete-intersection Calabi-Yau threefolds, expounding upon the presentation in Remark 4.11 of \cite{GP}.  These results are well-known
but we have been unable to find a suitable reference.
We first have the following  result:

\begin{lemma}\label{small res} Let $X$ be a projective threefold with only ODP singularities.  Suppose further that there exists a smooth Weil divisor $S$ passing through all of the ODP's.  Then there exists a projective small crepant resolution $\pi: \tilde{X}\rightarrow X$ obtained by blowing up $S$.  Moreover, the restriction $\pi: \pi^{-1}(S)\rightarrow S$ is the blow-up of $S$ at the smooth points 
of $S$ located at ODP's of $X$.  Finally, the natural map $\pi^*\Omega_X\rightarrow \Omega_{\tilde{X}}$ is injective.
\end{lemma}
\begin{proof} The blowup of $X$ along $S$ is projective. Its smoothness and the other claims
can be verified locally in the strong topology.  Since 3-fold ordinary double points have embedding dimension 4, we may represent $X$ locally around an ordinary double point as $f(x,y,z,w)=0$ in $\mathbb C^4$, where $f$ has vanishing partials at the origin and nonsingular Hessian.  Since $S$ is smooth, we may take local coordinates such that $S$ is locally represented by $x=z=0$ and $f=xQ+zP$ since $S$ passes through the singular point.  The first two claims of the lemma follow from applying the Jacobian criterion on the blow-up along $S$, and checking the restriction of the blow-up morphism to $S$. 

For the last statement, one finds from the description above that $X$ can locally be represented as Spec $\mathbb C[x,y,z,w]/(xy-zw)$, $\tilde{X}$ as Spec $\mathbb C[x,w,v]$, while the map $\pi$ corresponds to the ring homomorphism $R=\mathbb C[x,y,z,w]/(xy-zw)\rightarrow T=\mathbb C[x,w,v]$ with $x\mapsto x,y\mapsto vw,z\mapsto xv,w\mapsto w$.  From this description it follows immediately that the map on differentials $\Omega_{R/\mathbb C}\otimes_R T\rightarrow \Omega_{T/\mathbb C}$ is injective.  
\end{proof}

Let $X$ be a complete intersection Calabi-Yau threefold in $\mathbb P^{n+3}$ with only nodal singularities, given by $f_1=...=f_n=0$ of degrees $d_i$.  We assume that we are in the situation of the lemma so that there exists a small resolution $\pi:\tilde{X}\rightarrow X$.  Our first result guarantees that such an $\tilde{X}$ will indeed be a bona fide smooth Calabi-Yau threefold:

\begin{proposition} Let $X$ be a complete intersection Calabi-Yau threefold with only nodal (rational) singularities, admitting a small resolution $\pi:\tilde{X}\rightarrow X$.  Then $\tilde{X}$ is a nonsingular Calabi-Yau threefold, that is $\omega_{\tilde{X}}\cong \mathcal O_{\tilde{X}}$ and $h^i(\tilde{X},\mathcal O_{\tilde{X}})=0$ for $i=1,2$.
\end{proposition} 
\begin{proof} Since $\pi$ is a small resolution and complete intersections are Gorenstein, $\pi$ is crepant, i.e. $\pi^*\omega_X\cong \omega_{\tilde{X}}$.  But $X$ is Calabi-Yau, so $\omega_X\cong \mathcal O_X$ and thus $\omega_{\tilde{X}}\cong \mathcal O_{\tilde{X}}$.  For the statement on cohomology, we get the exact sequence of low degree terms coming from the Leray spectral sequence, \[0\rightarrow H^1(X,\pi_*\mathcal O_{\tilde{X}})\rightarrow H^1(\tilde{X},\mathcal O_{\tilde{X}})\rightarrow H^0(X,R^1\pi_*\mathcal O_{\tilde{X}}).\]  Integral complete intersections are normal so that $\pi_*\mathcal O_{\tilde{X}}\cong \mathcal O_X$, and because our singularities are rational we have $R^i\pi_*\mathcal O_{\tilde{X}}=0$ for $i>0$.  Thus $H^1(\tilde{X},\mathcal O_{\tilde{X}})\cong H^1(X,\mathcal O_X)=0$, and by Serre duality and triviality of $\omega_{\tilde{X}}$ we have $h^2(\tilde{X},\mathcal O_{\tilde{X}})=h^1(\tilde{X},\mathcal O_{\tilde{X}})=0$.
\end{proof}

We will now describe the calculation for the Hodge numbers $h^{1,1}(\tilde X)$ and $h^{2,1}(\tilde X)$.
Consider the tangent space to the deformation space of $X$, $\text{Def}_X(\mathbb C[\epsilon])\cong \text{Ext}^1_{\mathcal O_X}(\Omega_X^1,\mathcal O_X)$. It fits into an exact sequence isomorphic to the local-to-global exact sequence for Exts: \[0\rightarrow H^1(X,\mathcal T_X)\rightarrow \text{Ext}^1_{\mathcal O_X}(\Omega^1_X,\mathcal O_X)\rightarrow H^0(X,\mathcal T^1_X)\rightarrow H^2(X,\mathcal T_X),\] where \[\mathcal T_X\cong Hom_X(\Omega_X^1,\mathcal O_X),\mathcal T_X^1\cong Ext_{\mathcal O_X}^1(\Omega_X^1,\mathcal O_X),\] are the tangent sheaf of $X$ and Schlessinger-Lichtenbaum's $\mathcal T^1$ sheaf, respectively (see \cite{Ser} for proofs of these and all general deformation theoretic results).  Notice that $\mathcal T_X^1$ is a skyscraper sheaf with support precisely at the nodes, and over each node it is a copy of $\mathbb C$ since the deformation space of a node is one-dimensional.  

Since $\tilde{X}$ is nonsingular, the tangent space to the deformation space of $\tilde{X}$ is given by $H^1(\tilde{X},\mathcal T_{\tilde{X}})$.  Since $\tilde{X}$ is a Calabi-Yau threefold, we have $\bigwedge^2 \Omega_{\tilde{X}}^1\cong \mathcal T_{\tilde{X}}\otimes \omega_{\tilde{X}}\cong \mathcal T_{\tilde{X}}$ so that $h^{2,1}=h^1(\tilde{X},\bigwedge^2 \Omega_{\tilde{X}}^1)=h^1(\tilde{X},\mathcal T_{\tilde{X}})$.  We have the following result which allows us to use the deformation theory for $X$ to calculate this Hodge number for $\tilde{X}$:

\begin{proposition} With $\tilde{X},X$ as above, $H^1(\tilde{X},\mathcal T_{\tilde{X}})\cong H^1(X,\mathcal T_X)$.
\end{proposition}
\begin{proof} The resolution $\pi$ has as its exceptional locus a finite disjoint union of smooth rational curves $\{E_i\}$.  We use the natural exact sequence \[0\rightarrow \pi^*\Omega_X\rightarrow \Omega_{\tilde{X}}\rightarrow \Omega_{\tilde{X}/X}\rightarrow 0,\] which is left exact from Lemma \ref{small res}.  Notice that $\Omega_{\tilde{X}/X}\cong \bigoplus \omega_{E_i}$.  Note that $Hom(\bigoplus \omega_{E_i},\mathcal O_{\tilde{X}})=0$, since $\bigoplus \omega_{E_i}$ is a torsion sheaf, and $Ext^1(\bigoplus \omega_{E_i},\mathcal O_{\tilde{X}})=0$ from duality for projective schemes and the fact that $\omega_{\tilde{X}}\cong \mathcal O_{\tilde{X}}$.  Thus upon dualizing the above exact sequence we find that $\mathcal T_{\tilde{X}}= Hom(\Omega_{\tilde{X}},\mathcal O_{\tilde{X}})\cong Hom(\pi^*\Omega_X,\mathcal O_{\tilde{X}})$.  Pushing forward and using a form of the projection formula we get that \[\pi_*\mathcal T_{\tilde{X}}\cong \pi_* Hom(\pi^*\Omega_X,\mathcal O_{\tilde{X}})\cong Hom(\Omega_X,\pi_* \mathcal O_{\tilde{X}})\cong \mathcal T_X,\] where as before we note that $\pi_*\mathcal O_{\tilde{X}}\cong \mathcal O_X$.  We'll be done if we can show that $H^1(\tilde{X},\mathcal T_{\tilde{X}})\cong H^1(X,\pi_* \mathcal T_{\tilde{X}})$.

From the exact sequence of low terms arising from the Leray spectral sequence we get that \[0\rightarrow H^1(X,\pi_* \mathcal T_{\tilde{X}})\rightarrow H^1(\tilde{X},\mathcal T_{\tilde{X}})\rightarrow H^0(X,R^1\pi_*\mathcal T_{\tilde{X}}),\] so it suffices to prove that $R^1\pi_*\mathcal T_{\tilde{X}}=0$.  Clearly $(R^1 \pi_*\mathcal T_{\tilde{X}})_x=0$ for $x$ a nonsingular point.  We may check the vanishing on each node separately, so consider the exact sequence on $E_i$, \[0\rightarrow \mathcal T_{E_i} \rightarrow \mathcal T_{\tilde{X}}|_{E_i}\rightarrow \mathcal N_{E_i/\tilde{X}}\rightarrow 0,\] where $E_i$ is the exceptional curve above the node $x$.  Since $E_i\cong \mathbb P^1$ we have $\mathcal T_{E_i}\cong \mathcal O_{\mathbb P^1}(2)$, so $h^1(E_i,\mathcal T_{E_i})=0$.  By Remark 5.2 in \cite{Pagoda} we have $\mathcal N_{E_i/\tilde{X}}\cong \mathcal O_{\mathbb P^1}(-1)^{\oplus 2}$, so $h^1(E_i,\mathcal N_{E_i/\tilde{X}})=0$.  From the exact sequence above we see that $h^1(E_i,\mathcal T_{\tilde{X}}|_{E_i})=0$.  By using the Theorem on Cohomology and Base Change (Theorem III.12.11 in \cite{Har}) we can conclude that $R^1\pi_*\mathcal T_{\tilde{X}}\otimes k(x)=0$ and thus $(R^1\pi_* \mathcal T_{\tilde{X}})_x=0$ by Nakayama's Lemma.  
\end{proof}

The previous proposition implies that $h^{2,1}(\tilde{X})$ is the dimension of the kernel of the map $\text{Ext}^1_{\mathcal O_X}(\Omega^1_X,\mathcal O_X)\rightarrow H^0(X,Ext_{\mathcal O_X}^1(\Omega_X^1,\mathcal O_X))$, so we must understand this map better.

Denote by $S$ the homogeneous coordinate ring of $X$, so that \[S\cong \mathbb C[x_0,...,x_{n+3}]/(f_1,...,f_n).\]  Then since $X$ is a complete intersection, the conormal exact sequence is left-exact as well, \[0\rightarrow \mathcal I/\mathcal I^2\rightarrow \Omega_{\mathbb P^{n+3}}|_X\rightarrow \Omega_X\rightarrow 0,\]and $\mathcal I/\mathcal I^2\cong \bigoplus \mathcal O_X(-d_i)$.  Thus $\text{Hom}(\mathcal I/\mathcal I^2,\mathcal O_X)\cong \bigoplus H^0(X,\mathcal O_X(d_i))\cong \bigoplus S_{d_i}$ since $X$ is projectively normal, and \[\text{Ext}^1(\Omega_{\mathbb P^{n+3}}|_X,\mathcal O_X)\cong H^1(X,\mathcal T_{\mathbb P^{n+3}}|_X)=0\] from the Euler exact sequence.  From the cohomology of the conormal exact sequence we see $\text{Ext}^1(\Omega_X,\mathcal O_X)$ is the cokernel of the natural map \[\text{Hom}(\Omega_{\mathbb P^{n+3}}|_X,\mathcal O_X)\rightarrow \text{Hom}(\mathcal I/\mathcal I^2,\mathcal O_X)\cong \bigoplus S_{d_i}.\]  Applying global Hom to the Euler exact sequence restricted to $X$ and using the usual connection of Hom with $H^0$ we get, \[0\rightarrow H^0(X,\mathcal O_X)\rightarrow H^0(X,\mathcal O_X(1))^{n+4}\rightarrow \text{Hom}(\Omega_{\mathbb P^{n+3}}|_X,\mathcal O_X)\rightarrow 0.\]  Thus we can represent $\text{Ext}^1(\Omega_X,\mathcal O_X)$ as the cokernel of the map \[(S_1)^{n+4}=H^0(X,\mathcal O_X(1))^{n+4}\rightarrow \bigoplus H^0(X,\mathcal O_X(d_i))=\bigoplus S_{d_i},\] given by the matrix 
\[\left ( \begin{matrix}
\partial{f_1}/\partial{x_0} &  \ldots & \partial{f_1}/\partial{x_{n+3}}\\
\vdots & \ddots & \vdots\\
\partial{f_n}/\partial{x_0} &\ldots & \partial{f_n}/\partial{x_{n+3}}
\end{matrix} \right ).\]

To understand the kernel of the map \[\text{Ext}^1(\Omega_X,\mathcal O_X)\rightarrow H^0(X,Ext^1(\Omega_X,\mathcal O_X))\] better, we apply sheaf $Hom$ to the conormal exact sequence to get \[0\rightarrow Hom(\Omega_X,\mathcal O_X)\rightarrow \mathcal T_{\mathbb P^{n+3}}|_X\rightarrow \mathcal N_{X\backslash \mathbb P^{n+3}}\rightarrow Ext^1(\Omega_X,\mathcal O_X)\rightarrow 0.\]  From the discussion in the above paragraph, we find that the map in question is induced by the map on global sections \[H^0(X,\mathcal N_{X\backslash \mathbb P^{n+3}})\rightarrow H^0(X,Ext^1(\Omega_X,\mathcal O_X)),\] where of course as before $H^0(X,\mathcal N_{X\backslash \mathbb P^{n+3}})=\bigoplus S_{d_i}$.  To find the kernel of this map, we consider the map of sheaves on each affine piece, $U_i=\{x_i\neq 0\}$, $i=0,...,n+3$, so that if $(g_1,...,g_n) \in \bigoplus S_{d_i}$ goes to zero in $H^0(X,Ext^1(\Omega_X,\mathcal O_X))$, then its restrictions to $U_i$, upon which $\mathcal N_{X\backslash\mathbb P^{n+3}}$ is trivial and the restrictions are $(x_i^{-d_1}g_1,...,x_i^{-d_n}g_n)$ in the affine coordinates $T_j=x_j/x_i$, also go to zero in $H^0(U_i,Ext^1(\Omega_X,\mathcal O_X))$ for each $i$.  Since these $U_i$'s are affine and all of the sheaves in question are coherent (and thus quasi-coherent), taking sections on these open affines keeps the sequence exact, and thus we can calculate explicitly when this $n$-tuple goes to zero in $H^0(U_i,Ext^1(\Omega_X,\mathcal O_X))$.  Specifically, it must come from $H^0(U_i,\mathcal T_{\mathbb P^{n+3}}|_X)$ and thus from the affine version of the dual of the conormal exact sequence, we get that there exist $h^i_k(T_0,...,\hat{T_i},...,T_{n+3})\in \mathcal O_X|_{U_i}$ such that \[x_i^{-d_j}g_j=\sum_{k\neq i} x_i^{-d_j+1} \frac{\partial{f_j}}{\partial{x_k}} h^i_k,\] where we've dehomogenized the $g_j$'s and the $f_j$'s to their affine counterparts.  But then we can homogenize to get \[x_i^{\sum_{k\neq i} \deg h^i_k} g_j=\sum_{k\neq i} (x_i^{1+\sum_{s\neq k,i} \deg h^i_s} \frac{\partial{f_j}}{\partial{x_k}})[x_i^{\deg h^i_k}h^i_k(T_0,...,\hat{T_i},...,T_{n+3})],\] where now $x_i^{\deg h^i_k}h^i_k(T_0,...,\hat{T_i},...,T_{n+3})=p^i_k$ are homogeneous polynomials.  Consider the $n\times n$ minors of the matrix \[\left ( \begin{matrix} g_1&
\partial{f_1}/\partial{x_0} &  \ldots & \partial{f_1}/\partial{x_{n+3}}\\
\vdots & \vdots & \ddots & \vdots\\
g_n &\partial{f_n}/\partial{x_0} &\ldots & \partial{f_n}/\partial{x_{n+3}}
\end{matrix} \right ).\]  Then from the above description we find that $x_i^{\sum_{k\neq i} \deg h^i_k}$ times each of these minors is in the homogeneous ideal of the singular locus.  We get such a description for each $i$ and since the homogeneous ideal is saturated we in fact find that the $n\times n$ minors themselves are in the homogeneous ideal of the singular locus.  Conversely, we show that if a pair $(g_1,...,g_n)\in \bigoplus S_{d_i}$ such that the $n\times n$ minors of \[\left ( \begin{matrix} g_1&
\partial{f_1}/\partial{x_0} &  \ldots & \partial{f_1}/\partial{x_{n+3}}\\
\vdots &\vdots & \ddots & \vdots\\
g_n &\partial{f_n}/\partial{x_0} &\ldots & \partial{f_n}/\partial{x_{n+3}}
\end{matrix} \right )\]  are in the homogeneous ideal of the singular locus, then $(g_1,...,g_n)$ is in the kernel of the map in question.  Restricting again to the affine open $U_i$, we get that the corresponding minors of the dehomogenizations are in the affine ideal of the singular locus of $U_i$.  Let $B_i$ denote the affine coordinate ring of $Sing(X)\cap U_i$.  Then since $Ext^1(\Omega_X,\mathcal O_X)|_{U_i}$ has support on precisely this set, its module structure is unchanged by considering it as a $B_i$ module.  Thus tensoring our usual right exact sequence, \[\mathcal T_{\mathbb P^{n+3}}|_{U_i}\rightarrow \mathcal N_{X\backslash \mathbb P^{n+3}}\rightarrow Ext^1(\Omega_X,\mathcal O_X)|_{U_i}\rightarrow 0,\] by $B_i$ we get that $Ext^1(\Omega_X,\mathcal O_X)|_{U_i}$ may be viewed as the cokernel of the map \[B_i^{n+3}\rightarrow B_i^n,\] given by the dehomogenization of the Jacobian.  Localizing at each of the isolated points of $Sing(X)\cap U_i$ and using Nakayama's lemma, we may just mod out by the maximal ideal $\mathfrak m_p$ corresponding to such a singular point $p$.  Since these points are ordinary double points, $Ext^1(\Omega_X,\mathcal O_X)\otimes \kappa(p)=\mathbb C$, so the Jacobian as a map of vector spaces has rank $n-1$, that is there are $n-1$ columns of the Jacobian which are linearly independent.  Since the minor involving the column $(g_1,...,g_n)$, reduced mod $\mathfrak m_p$, with these $n-1$ columns of the Jacobian vanishes, we find that $(g_1,...,g_n)$ must be a linear combination of these columns, i.e. it's in the image of the Jacobian.  Since this is each true for each singular point $p$, we find that $(g_1,...,g_n)$ goes to zero in $H^0(U_i,Ext^1(\Omega_X,\mathcal O_X))$ for each $i$, and thus in $H^0(X,Ext^1(\Omega_X,\mathcal O_X))$.  

Using the description in the previous paragraph, of $h^{2,1}(\tilde{X})$ as the dimension of the space of $n$-tuples $(g_1,...,g_n)$ representing an element of $\text{Ext}^1(\Omega,\mathcal O_X)$ such that the $n\times n$ minors of \[\left ( \begin{matrix} g_1&
\partial{f_1}/\partial{x_0} &  \ldots & \partial{f_1}/\partial{x_{n+3}}\\
\vdots & \vdots &\ddots & \vdots\\
g_n &\partial{f_n}/\partial{x_0} &\ldots & \partial{f_n}/\partial{x_{n+3}}
\end{matrix} \right ),\] are in the homogeneous ideal of the singular locus, one may actually calculate the Hodge number $h^{2,1}(\tilde{X})$ using \textit{Macaulay2}.  If one knows the topological Euler characteristic of $\tilde{X}$, as we did in our (2,4) and (3,3) constructions, then one may calculate the Picard number, $h^{1,1}$, as well.

\section{Appendix 2: Enriques surfaces embedded into Calabi-Yau threefolds}
 In this appendix, we prove that if a smooth Calabi-Yau threefold $X$ contains an Enriques surface, then the Enriques surface appears as a double fibre of a fibration of $X$ by K3 surfaces.  This was observed in \cite{MP} for Calabi-Yau threefolds $Q=K3\times E/\sigma$, where $E$ is an elliptic curve and the involution $\sigma$ acts on $K3$ by a fixed-point free involution and on $E$ by translation by a $2$-torsion point.  In their example, projection to the second factor induces a K3 fibration with 4 double Enriques fibers.  See \cite{MP} for details and a discussion of the Gromov-Witten theory of such Calabi-Yau threefolds.

\begin{proposition}\label{CYfibration} Let $X$ be a Calabi-Yau threefold, and $i:E\subset X$ an Enriques surface.  Then $|2E|$ is a base-point free linear pencil that induces a fibration $\pi:X\rightarrow \mathbb P^1$ whose generic fibre is a K3 surface and in which $E$ appears as a double fibre.
\end{proposition}

\begin{proof} Consider the short exact sequence on $X$, 
$$
0\rightarrow {\mathcal O}_X\rightarrow {\mathcal O}_X(E)
\rightarrow i_*{\mathcal O}_E(E)\rightarrow 0
.$$  
Since $X$ is Calabi-Yau, we have by adjunction that $\omega_E\cong\omega_X(E)|_E\cong \mathcal O_E(E)$, where $\omega_X,\omega_E$ are the canonical bundles of $X$ and $E$, respectively.  We use the long exact sequence in cohomology to calculate cohomology 
of ${\mathcal O}_X(E)$.
Since $\omega_E$ has no global sections, we see that $h^0(X,\mathcal O_X(E))=1$.  Since $h^1(X,\mathcal O_X)=h^1(E,\omega_E)=0$, we also see that $h^1(X,\mathcal O_X(E))=0$.  Furthermore, we have $h^2(X,\mathcal O_X)=0$, $h^2(E,\omega_E)=h^0(E,\mathcal O_E)=1$ by Serre duality, and $h^3(X,\mathcal O(E))=h^0(X,\mathcal O(-E))=0$ by Serre duality and the fact that $-E$ is not effective.  Finally $h^3(X,\mathcal O_X)=1$ 
implies  $h^2(X,\mathcal O(E))=0$.

Now consider the short exact sequence, 
$$0\rightarrow \mathcal O_X(E)\rightarrow \mathcal O_X(2E)\rightarrow 
i_* \mathcal O_E\rightarrow 0,$$
where we have used $\mathcal O_E\cong \omega_E^2\cong  \mathcal O_E(2E)$.  From 
$h^1(\mathcal O_X(E))=0$
we conclude that  $h^0(X,\mathcal O_X(2E))=2$ and  $h^1(X,\mathcal O_X(2E))=0$. 

To see that the linear system $|2E|$ is base-point free, note that its only base-points must be along $E$ itself.  However 
a section of $\mathcal O_X(2E)$ that maps nontrivially to $H^0(i_* \mathcal O_E)$
does not vanish on $E$.
Thus we get a morphism $\pi:X\rightarrow \mathbb P^1$ with $E$ a double fibre. 
By Generic Smoothnes (Theorem III.10.7 in \cite{Har}) the generic fibre of $\pi$ is smooth.  So let $D\in |2E|$ be such a generic fibre.  Then by adjunction $\omega_D\cong\mathcal O_D(2E)\cong \mathcal O_D$ since $D$ and $2E$ are distinct fibers.  From the exact sequence, \[0\rightarrow \mathcal O_X\rightarrow \mathcal O_X(2E)\rightarrow i_*\mathcal O_D(2E)\rightarrow 0,\] and the fact that $\omega_D\cong \mathcal O_D(2E)$ we see that $h^1(D,\omega_D)=0$ since by above $h^1(X,\mathcal O_X(2E))=0=h^2(X,\mathcal O_X)$.  Thus $D$ is a K3 surface.
\end{proof}

\section{Appendix 3: Linear algebra lemma}\label{linalg}
In this section we establish some linear algebra results which are used in
Section \ref{secwp}. We will state the results for arbitrary $n\geq 1$ although only the $n=2$
case will be used in this paper.

\begin{proposition}\label{tn}
Let $M$ be a $2n\times 2n$ matrix with complex coefficients.
Let $T$ be a $2n\times n$ complex matrix of rank $n$ such that 
$$
T^*MT=0.
$$
Then for any $2n\times n$ complex matrix $S$ we have 
$$
\det (S^*MT)\det(T^*MS) = (-1)^n\det (M) \det(T\vert S)^2
$$
where $T\vert S$ is a $2n\times 2n$ matrix made from $T$ and $S$.
\end{proposition}

\begin{proof}
The transformation $(M,T,S)\mapsto
(U^*MU,U^{-1}T,U^{-1}S)$ 
for an invertible $2n\times 2n$ matrix $U$  does not alter the two sides of the claim.
Thus we may assume that 
$$
T=\left(
\begin{array}{c}
{\rm Id}_{n}\\
0
\end{array}
\right).
$$
Then $T^*MT=0$ means that 
$$
M=\left(
\begin{array}{cc}
0&M_1\\
M_2&M_3
\end{array}
\right)
$$
for some $n\times n$ matrices $M_i$. If 
$$
S=\left(
\begin{array}{c}
S_1\\
S_2
\end{array}
\right),
$$
then 
$$
S^*MT=S_2M_2,~T^*MS=M_1S_2,~\det (T\vert S)=\det S_2
$$
and the claim follows easily.
\end{proof}

We will also need a particular case of this statement.
\begin{proposition}\label{skew}
Let $M$ and $T$ be as above. Assume in addition that $M^*=-M$. 
Then 
$$
\det(S^*MT)=(-1)^n\det(T\vert S) {\rm Pf}(M).
$$
\end{proposition}

\begin{proof}
The proof is analogous.
\end{proof}

\begin{proposition}\label{kernel}
If $T^*MT=0$, ${\rm rk}M=2n-1$ and $\det (S^*MT)=0$, then 
the right kernel of $M$ lies in  the span of columns of $T$.
\end{proposition}

\begin{proof} We again transform $T$ into a standard form. We observe that
since $\rk M=2n-1$, the matrix $M_1$ is nondegenerate, so the last $n$ components
of the right kernel of $M$ are zero.
\end{proof}.

\end{document}